\documentclass[a4paper,11pt]{article}
\usepackage[T1]{fontenc}
\usepackage[utf8]{inputenc}
\usepackage{amsmath,amsthm,amssymb,color,graphicx,bm,mathtools}
\usepackage{algorithm}
\usepackage{algorithmic}
\usepackage{layaureo}
\usepackage{microtype}
\usepackage{emptypage}
\usepackage{indentfirst}
\usepackage{CJKutf8}
\usepackage{enumitem}
\usepackage[T1]{fontenc}
\usepackage{lmodern}
\usepackage[affil-it]{authblk}
\usepackage[english]{babel}
\usepackage{blindtext}

\usepackage[all, pdf]{xy}
\usepackage{subfig}
\usepackage{enumitem}
\usepackage{color}
\usepackage{graphicx,booktabs}
\usepackage[bookmarks,colorlinks=true,linkcolor=blue]{hyperref} 

\newcommand{\numberset}{\mathbb}
\newcommand{\N}{\numberset{N}} 
\newcommand{\R}{\numberset{R}} 
\renewcommand{\phi}{\varphi} 
\renewcommand{\chi}{\mathcal{X}} 
\newcommand{\Hst}{\tilde{H}^s}
\renewcommand{\epsilon}{\varepsilon}

\newtheorem{theorem}{Theorem}

\newtheorem{definition}[theorem]{Definition}
\newtheorem{lemma}[theorem]{Lemma}

\newtheorem{remark}[theorem]{Remark}

\theoremstyle{definition}
\newtheorem{example}[theorem]{Example}

\newenvironment{system}%
{\left\lbrace\begin{aligned}}%
{\end{aligned}\right.}

\begin{document}

\title{Weak solutions for nonlinear waves in adhesive phenomena}
\author[1]{Mauro Bonafini}
\author[2]{Van Phu Cuong Le}

\affil[1]{Institut f\"ur Informatik, Georg-August-Universit\"at G\"ottingen, Germany, e-mail: bonafini@cs.uni-goettingen.de}
\affil[2]{Dipartimento di Informatica, Universit\`a di Verona, Italy, e-mail: vanphucuong.le@univr.it}
\date{\today}

\maketitle
\begin{abstract}
	We discuss a notion of weak solution for a semilinear wave equation that models the interaction of an elastic body with a rigid substrate through an adhesive layer, relying on results in $\cite{MCMO}$. Our analysis embraces the vector-valued case in arbitrary dimension as well as the case of non-local operators (e.g. fractional Laplacian).
\end{abstract}
\textbf{Keywords:} nonlinear waves, adhesive phenomena, non-local wave equations, nonlinear elasticity 

\textbf{2020 Mathematics Subject Classification:} 35L05, 35L15.
\section{Introduction}
In recent years, there have been many works devoted to adhesion phenomena arising from biophysics and engineering (see for instance $\cite{Isaraela, Zhao, Ghatak, Oyha}$ and references therein). A rigorous mathematical description of such phenomena is quite challenging, mainly because of the complex underlying mechanisms at both microscopic and macroscopic levels. In order to reproduce some essential features of these processes, increasingly accurate mathematical models are being proposed (see for instance  $\cite{KellerBurridge, Maddalena1, Maddalena2, CocliteFlorioLigaboMaddalena17, Yoshiho}$ and references therein).

Of particular interest is the study of the dynamic of an elastic body glued to a rigid substrate through an adhesive layer. Consider, for example, a vibrating string oscillating in and out a glue layer. The position of the evolving string is described by the scalar displacement $u:\,[0, T]\times[0,L]\to\R$, for $T,L > 0$.
In the absence of an adhesive region, the dynamic is typically described via a free wave equation, with either Dirichlet or Neumann boundary conditions. The presence of adhesive regions can then be described via a ``forcing'' potential $W$. Hence, one is led to consider the following semilinear wave equation
\begin{equation}\label{eq:semi0}
	u_{tt} -u_{xx} + W'(u) = 0,
\end{equation} 
coupled with suitable initial and boundary conditions. In \cite{CocliteFlorioLigaboMaddalena17}, for example, the glue region is assumed to cover $[0,L]\times [-u^*, u^*]$, for $u^* > 0$, and $W$ is selected to be
\begin{equation}\label{potential}
W(u)=
\begin{system}
& u^{2}			&\text{if }& |u|\leq u^{*}, 	\\
& (u^{*})^{2}	&\text{if }& |u|> u^{*}.	\\
\end{system}
\end{equation}
In this setting, the equation reduces to a free wave equation as soon as the string moves outside the glue layer. Once we introduce \eqref{eq:semi0}, a suitable notion of weak solution that takes into account the discontinuity of $W'$ is immediately needed.

Motivated by the one dimensional model we just described, we consider here, for $m > 0$ and a potential $W \colon \R^m \to \R$, the generalized problem
\begin{equation}\label{eq:semi1}
\begin{system}
& u_{tt}+(-\Delta )^{s}u +\nabla_{u} W(u)= 0      	&\quad&\text{in } (0,T) \times \Omega, 						\\
& u(t,x) = 0             							&\quad&\text{in } [0,T] \times (\R^{d} \setminus \Omega), 	\\
& u(0,x) = u_0(x)                           		&\quad&\text{in } \Omega, 									\\
& u_t(0,x) = v_0(x)                         		&\quad&\text{in } \Omega, 									\\
\end{system}
\end{equation} 
where $\Omega\subset \R^d$ is an open bounded domain with Lipschitz boundary, $u_0$ and $v_0$ are suitable initial conditions, and  $(-\Delta )^{s}$ stands for the fractional Laplacian ($s=1$ provides the standard Laplacian). We are interested in possible notions of weak solutions for such a system.
As self-evident, any notion of solution for \eqref{eq:semi1} heavily hinges on the regularity of the potential $W$. On one hand, for regular $W$, i.e. $W \in C^{1,1}(\R^m)$ and non-negative, existence of suitably defined weak solutions has been provided in \cite[Theorem 3]{MCMO}. On the other hand, less is known for less regular potentials like \eqref{potential}. In this note, we aim to explore how far the notion of weak solution of \cite{MCMO} can reach, and discuss its limitations.

In Section \ref{pre}, we first recall the notion of fractional Laplacian and fractional Sobolev spaces, we state the model problem \eqref{eq:semi1} and in Definition \ref{def:weak} we recall the working notion of weak solution. In Section \ref{case1}, Theorem \ref{thm:main1}, we assume $W \in {C}^1(\R^m)$ to be bounded and to have a bounded uniformly continuous gradient, and prove existence of weak solutions for \eqref{eq:semi1} via an approximation argument. In Section \ref{case2}, through Example \ref{exa}, we discuss how the proposed notion of weak solution cannot be adequate to less regular settings. Eventually, in Theorem \ref{thm:main2}, we prove existence of weak solutions for potentials behaving like \eqref{potential} under the restrictive assumption of small initial data.

\section{Preliminaries and model problem}\label{pre}
Let $d, m \in \mathbb{N}$ and $s > 0$. We define the fractional Laplacian operator $(-\Delta)^{s}$ as the operator whose Fourier symbol is $|\xi|^{2s}$, i.e., for any $u\in L^{2}(\R^{d}; \R^{m})$, we set
$$\mathcal{F}(-\Delta)^{s}u=|\xi|^{2s}\mathcal{F}u$$
where $\mathcal{F}$ denotes the Fourier transform. We denote by $H^s$ the fractional Sobolev space of order $s$, which is defined as
\[
H^s(\R^d) := \left\{ u \in L^2(\R^d ;\, \R^{m}) \,:\, \int_{\R^d} (1+|\xi|^{2s})|\mathcal{F}u(\xi)|^2\,d\xi < +\infty \right\}.
\]
For $u,v \in H^s(\R^d)$, we consider the scalar product $[u, v]_{s} = \langle (-\Delta)^{s/2}u, (-\Delta)^{s/2}v \rangle_{L^2(\R^d; \R^m)}$, the corresponding semi-norm $[u]_s = \sqrt{[u,u]_s} = ||(-\Delta)^{s/2}u||_{L^2(\R^d; \R^{m})}$ and the norm $||u||_s^2 = ||u||_{L^2(\R^d;\R^{m})}^2 + [u]_s^2$. For $\Omega \subset \R^d$ an open bounded set with Lipschitz boundary, we define
\[
\tilde{H}^s(\Omega) := \{ u \in H^s(\R^d;\R^{m}) \,:\, u = 0 \text{ a.e. in } \R^d\setminus \Omega \},
\]
and the corresponding dual space $H^{-s}(\Omega) := (\Hst(\Omega))^*$. The set of smooth compactly supported functions $C^{\infty}_{c}(\R^{d};\R^{m})$ is dense in $H^{s}(\R^{d})$ (see \cite{Luc}). We eventually recall and quickly prove the following embedding (based on \cite{Luc}).
\begin{lemma}\label{embedding}
	Let $s>0$, $2s>d$ and $u \in H^{s}(\R^{d})$. Then, $u \in C^{0}(\R^{d};\R^{m})$ and there exists a constant $C$ independent of $u$ such that
	\begin{equation}\label{embedding0}
	||u||_{C^{0}(\R^{d};\R^{m})}\leq C||u||_{H^s(\R^d)}.
	\end{equation}
\end{lemma}
\begin{proof}
	Let $\mathcal{S}(\R^{d};\R^{m})$ be the Schwartz space of rapidly decaying functions and fix $f \in \mathcal{S}(\R^{d};\R^{m})$. One has
	\begin{equation*}\label{uniqueness1}
	\begin{aligned}
	(2\pi)^{\frac{d}{2}}|f(x)|=\left|(2\pi)^{\frac{d}{2}}\mathcal{F}^{-1}(\hat{f})(x)\right|&=\left|\int_{\R^{d}}e^{ix\xi }\mathcal{F}f(\xi)d\xi\right|\leq \int_{\R^{d}}|\mathcal{F}f(\xi)|(1+|\xi|^{s})\frac{1}{1+|\xi|^{s}}d\xi\\
	&\leq \left( \int_{\R^{d}}\frac{1}{(1+|\xi|^{s})^{2}}d\xi \right)^{\frac{1}{2}} \left( \int_{\R^{d}}|\mathcal{F}f(\xi)|^{2}(1+|\xi|^{s})^{2}d\xi \right)^{\frac{1}{2}}\\
	&\leq \sqrt{2}\left( \int_{\R^{d}}\frac{1}{(1+|\xi|^{s})^{2}}d\xi \right)^{\frac{1}{2}} \left( \int_{\R^{d}}|\mathcal{F}f(\xi)|^{2}(1+|\xi|^{2s})d\xi \right)^{\frac{1}{2}}\\
	&\leq \sqrt{2}\left( \int_{\R^{d}}\frac{1}{(1+|\xi|^{s})^{2}}d\xi \right)^{\frac{1}{2}} ||f||_{H^s(\R^d)}.
	\end{aligned}
	\end{equation*}
	Since we consider $2s>d$, $\int_{\R^{d}}\frac{1}{(1+|\xi|^{s})^{2}}d\xi$ is finite. Thus, we obtain that 
	\begin{equation}\label{embedding2}
	||f||_{C^{0}(\R^{d};\R^{m})}\leq C||f||_{H^s(\R^d)},
	\end{equation}
	which is exactly \eqref{embedding0}. Fix now any  $u\in H^s(\R^d)$. By the density of $\mathcal{S}(\R^{d};\R^{m})$ in $H^s(\R^d)$, there exists a sequence $\lbrace f_{n} \rbrace_{n} \subset \mathcal{S}(\R^{d};\R^{m})$ such that $f_{n}$ converges to $u$ in $H^s(\R^d)$. In particular, $\lbrace f_{n} \rbrace_n$ is a Cauchy sequence in $H^s(\R^d)$ and by $\eqref{embedding2}$ we obtain
	\begin{equation}\label{embedding3}
	||f_{k}-f_{l}||_{C^{0}(\R^{d};\R^{m})}\leq C||f_{k}-f_{l}||_{H^s(\R^d)} \quad \text{for any } k,l\in \N.
	\end{equation}
	Inequality $\eqref{embedding3}$ amounts to say that $\lbrace f_{n} \rbrace_n$ is a Cauchy sequence in $C^{0}(\R^{d};\R^{m})$, hence $u \in C^{0}(\R^{d};\R^{m})$, $f_n \to u$ uniformly as $n \to \infty$ and
	\begin{equation*}
	||u||_{C^{0}(\R^{d};\R^{m})}\leq C||u||_{H^s(\R^d)}.
	\qedhere
	\end{equation*}
\end{proof}
\subsection{Model problem.}
For an open bounded set $\Omega \subset \R^d$ with Lipschitz boundary and a potential $W\colon \R^m \to [0, \infty)$ (whose regularity we specify later on), we look for a solution $u\,:\,[0, T]\times \Omega \to \R^{m}$ of
\begin{equation}\label{eq:smlwaves}
\begin{system}
& u_{tt} + (-\Delta)^s u +\nabla_{u}W(u)= 0             	&\quad&\text{in } (0,T) \times \Omega                                   	\\
& u(t,x) = 0                                &\quad&\text{in } [0,T] \times (\R^d \setminus \Omega)                      \\
& u(0,x) = u_0(x)                           &\quad&\text{in } \Omega                                                    \\
& u_t(0,x) = v_0(x)                         &\quad&\text{in } \Omega                                                    \\
\end{system}
\end{equation}
with initial data $u_0 \in \Hst(\Omega)$ and $v_0 \in L^2(\Omega; \R^m)$ (we intend that $v_{0}=0 \mbox{ in } \R^{d} \setminus \Omega$). For $m = d$ one can conventionally interpret $u$ as the  displacement of an elastic body (see \cite[Section 2]{CocliteFlorioLigaboMaddalena19}). A notion of weak solution for problem $\eqref{eq:smlwaves}$ can be given as follows.
\begin{definition}[Weak solution and energy]\label{def:weak}
	Let $T > 0$. We say $u = u(t,x)$ is a weak solution of $\eqref{eq:smlwaves}$ in $(0,T)$ if
	\begin{enumerate}
		\item
		$
		u \in L^\infty(0,T; \Hst(\Omega)) \cap W^{1,\infty}(0,T;L^2(\Omega))$ and $u_{tt} \in L^\infty(0,T;H^{-s}(\Omega)),
		$
		\item for all $\phi \in L^{1}(0,T;\Hst(\Omega))$
		\begin{equation}\label{eq:eqweak}
		\int_{0}^T \langle u_{tt}(t), \phi(t) \rangle dt + \int_{0}^T [ u(t), \phi(t) ]_{s} \, dt +\int_{0}^T \int_\Omega \nabla_{u} W(u(t)) \phi(t) \,dxdt = 0
		\end{equation}
		with
		\begin{equation}\label{eq:u0free}
		u(0,x)=u_{0} \quad \text{ and } \quad u_{t}(0,x)=v_{0}.
		\end{equation}
	\end{enumerate}
	The energy of $u$ is defined as
	\[
	E(u(t)) = \frac12 ||u_{t}(t)||^{2}_{L^{2}(\Omega)}+\frac12 [u(t)]_{s}^2+||W(u(t))||_{L^{1}(\Omega)} \quad \text{for } t \in [0,T].
	\]
\end{definition}

Existence of a weak solution in the sense of Definition \ref{def:weak} has been proved in \cite[Theorem 3(i)]{MCMO} for non-negative potentials $W \in C^1(\R^m)$ with Lipschitz continuous gradient.
We prove in the next section the existence of weak solutions under slightly relaxed assumptions, i.e. $W \in C^1(\R^m)$, non-negative, uniformly bounded, with a bounded and uniformly continuous gradient. Less regular potentials are then partially addressed in Section \ref{case2}, where limitations of the current approach are discussed in Example \ref{exa}, and existence of weak solutions is provided, for $2s > d$, under the restrictive assumption of small initial data.

\section{The case of continuous $\nabla W$}\label{case1}

This section is devoted to the proof of the following theorem.
\begin{theorem}\label{thm:main1}
	Let $W \in C^1(\R^m)$, and $W$ be non-negative. Assume there exists $K > 0$ such that $0 \leq W(y) \leq K$ and $0 \leq |\nabla W(y)| \leq K$ for all $y \in \R^m$, with $\nabla W$ uniformly continuous. Then, there exists a weak solution of $\eqref{eq:smlwaves}$ satisfying the energy inequality
	\begin{equation}\label{energyestimate1}
	E(u(t))\leq E(u(0)) \quad \text{for any $t\in [0, T]$}.
	\end{equation}
\end{theorem}
The hypotheses of Theorem \ref{thm:main1} reproduce prototypical settings modelling adhesive behaviours, as in \cite{Maddalena3}, where the forcing potential $W$ is expected to have no influence outside a bounded region, and thus we can assumed $W$ to be constant outside said region. Our proof relies on an approximating procedure and leverages existence results in \cite{MCMO} to provide existence of approximate weak solutions. The given regularity of the potential will then ensure that we can pass to the limit along the sequence of approximate solutions (see Step 3 in the proof below).
\begin{proof}[Proof of Theorem \ref{thm:main1}]
	$\,$
	
	\noindent
	\emph{Step 1. Construction of regularized approximate weak solutions.}
	Let us consider a family of non-negative potentials $(W_{\epsilon})_{\epsilon >0}$ in $C^{2}(\R^m)$ such that:
	\begin{itemize}
		\item[(i)] $W_{\epsilon}$ converges uniformly to $W$,
		\item[(ii)] $\nabla W_{\epsilon}$ converges uniformly to $\nabla W$,
		\item[(iii)] $\nabla W_{\epsilon}$ is Lipschitz continuous for each $\epsilon$.
	\end{itemize} 
	Leveraging the existence result in \cite[Theorem 3(i)]{MCMO}, for each $\epsilon >0$ there exists a weak solution $u^{\epsilon}$ of
	\begin{equation*}
	\begin{system}
	& u^{\epsilon}_{tt} + (-\Delta)^s u^{\epsilon} +\nabla_{u}W_\epsilon(u^{\epsilon})= 0  &\quad&\text{in } (0,T) \times \Omega       \\
	& u^{\epsilon}(t,x) = 0                                &\quad&\text{in } [0,T] \times (\R^d \setminus \Omega)                      \\
	& u^{\epsilon}(0,x) = u_0(x)                           &\quad&\text{in } \Omega                                                    \\
	& u^{\epsilon}_t(0,x) = v_0(x)                         &\quad&\text{in } \Omega                                                    \\
	\end{system}
	\end{equation*}
	in the sense of Definition \ref{def:weak}. In particular, we have
	\begin{equation}\label{Solutionestimate1}
	\int_{0}^T \langle u^\epsilon_{tt}(t), \phi(t) \rangle dt + \int_{0}^T [ u^\epsilon(t), \phi(t) ]_{s} \, dt + \int_{0}^T \int_\Omega \nabla_{u} W_\epsilon(u^\epsilon(t)) \phi(t) \,dxdt = 0
	\end{equation}
	for all $\phi \in L^{1}(0,T;\Hst(\Omega))$, and, for any $t \in [0,T]$, we have
	\begin{equation}\label{Eestimate1}
	\frac12 ||u^\epsilon_{t}(t)||^{2}_{L^{2}(\Omega)}+\frac12 [u^\epsilon(t)]_{s}^2+||W_\epsilon(u^\epsilon(t))||_{L^{1}(\Omega)} \leq \frac12 ||v_0||^{2}_{L^{2}(\Omega)}+\frac12 [u_0]_{s}^2+||W_{\epsilon}(u_0)||_{L^{1}(\Omega)}.
	\end{equation}
	
	\medskip
	\noindent
	\emph{Step 2. Existence of a cluster point.}
	Since $W_{\epsilon}$ converges uniformly to $W$ in $\R^{m}$ and $W$ is bounded, for sufficiently small $\epsilon$ we have a uniform bound on $W_{\epsilon}$. Thus, using $\eqref{Eestimate1}$, there exists a constant $C>0$ such that for any $t \in [0,T]$
	\begin{equation}\label{Eestimate2}
	E(u^\epsilon(t)) = \frac12 ||u^\epsilon_{t}(t)||^{2}_{L^{2}(\Omega)}+\frac12 [u^\epsilon(t)]_{s}^2+||W_\epsilon(u^\epsilon(t))||_{L^{1}(\Omega)} \leq C.
	\end{equation}
	From this energy bound, via standard compactness arguments, as done in \cite[Proposition 6]{MCMO}, we deduce that there exists $u \in L^\infty(0,T; \Hst(\Omega)) \cap W^{1,\infty}(0,T;L^2(\Omega))$ such that, up to a subsequence, as $\epsilon \to 0$, we have
	\begin{itemize}
		\item[(iv)] $u^\epsilon \to u$ in $C^0([0,T];L^2(\Omega))$,
		\item[(v)] $u_t^\epsilon \rightharpoonup^* u_t$ in  $L^\infty(0,T;L^2(\Omega))$,
		\item[(vi)] $u^\epsilon(t) \rightharpoonup u(t)$ in $\Hst(\Omega)$ for any $t \in [0,T]$,
		\item[(vii)] $u^\epsilon \rightharpoonup^* u$ in $L^\infty(0,T;\Hst(\Omega))$.
	\end{itemize}
	
	\medskip
	\noindent
	\emph{Step 3. Passage to the limit in the definition of weak solution.}
	In order to prove that $u$ is a weak solution we pass to the limit in \eqref{Solutionestimate1} as $\epsilon \to 0$. To do so, observe that
	\begin{itemize}
		\item $u_{tt}\in L^\infty(0,T;H^{-s}(\Omega))$ and  $u^\epsilon_{tt}  \rightharpoonup^* u_{tt}$ in $L^\infty(0,T;H^{-s}(\Omega))$
		
		Indeed, from $\eqref{Eestimate2}$, $\eqref{Solutionestimate1}$, and the uniform bound on $|\nabla W_{\epsilon}|$, we obtain that
		$u^{\epsilon}_{tt}$ is uniformly bounded in $L^\infty(0,T;H^{-s}(\Omega))$. This implies that $u^\epsilon_{tt}  \rightharpoonup^* u_{tt}$ in $L^\infty(0,T;H^{-s}(\Omega))$.
		
		\item $\nabla_{u} W_{\epsilon}(u^\epsilon) \rightharpoonup^* \nabla_{u} W(u)$ in $L^\infty(0,T;H^{-s}(\Omega))$
		
		Indeed, $u^\epsilon \to u$ for a.e.  $(x,t)\in (0,T)\times \Omega$ due to the convergence of $u^\epsilon$ to $u$ in $C^0([0,T];L^2(\Omega))$. Thus, since $\nabla W_{\epsilon}$ converges uniformly to $\nabla W$ in $\R^{m}$ and $\nabla W_{\epsilon}$ is uniformly bounded, by the dominated convergence theorem we conclude that
		$$
		\nabla_{u} W_{\epsilon}(u^\epsilon) \to \nabla_{u} W(u) \mbox{ in } L^2((0,T)\times \Omega).
		$$
		On the other hand, $\nabla_{u} W_{\epsilon}(u^\epsilon)$ is uniformly bounded in $L^\infty(0,T;H^{-s}(\Omega))$, therefore we can conclude that $\nabla_{u} W_{\epsilon}(u^\epsilon) \rightharpoonup^* \nabla_{u} W(u)$ in $L^\infty(0,T;H^{-s}(\Omega))$.
	\end{itemize}
	Thus, letting $\epsilon \to 0$ in \eqref{Solutionestimate1} we obtain
	\begin{equation}\label{solution}
	\int_{0}^T \langle u_{tt}(t), \phi(t) \rangle dt + \int_{0}^T [ u(t), \phi(t) ]_{s} \, dt +\int_{0}^T \int_\Omega \nabla_{u} W(u(t)) \phi(t) \,dxdt=0
	\end{equation}
	for all $\phi \in L^{1}(0,T;\Hst(\Omega))$. To conclude, observe that
	\begin{itemize}				
		\item $E(u(t))\leq E(u(0))$ for each $t\in [0, T]$
		
		From the fact that $u^{\epsilon}_{tt}$ is uniformly bounded in $L^\infty(0,T;H^{-s}(\Omega))$, we deduce that $u_{t}^{\epsilon} \to u_{t}$ in $C^{0}([0, T]; H^{-s}(\Omega))$. On the other hand, each $u^{\epsilon}_{t}(t)$ is uniformly bounded in $L^{2}(\Omega)$, thus we obtain that $u^\epsilon_{t}(t) \rightharpoonup u_{t}(t)$ in $L^2(\Omega)$ for each $t\in [0, T]$. 
		For the convergence of $W_{\epsilon}(u^\epsilon)$, let $t\in [0, T]$ and fix an arbitrary $\eta>0$. Since $W_{\epsilon}$ converges uniformly to $W$, for sufficiently small $\epsilon$ we obtain that
		\begin{equation}\label{uniformbound}
		|W_{\epsilon}(y)-W(y)|\leq \eta
		\end{equation}
		for any $y\in \R^m$. Hence,
		\begin{equation*}\label{convulotionestimate}
		\begin{aligned}
		\int_{\Omega}|W_{\epsilon}(u^{\epsilon}(t))-W(u(t))|dx&\leq \int_{\Omega}|W_{\epsilon}(u^{\epsilon}(t))-W(u^{\epsilon}(t))|dx+\int_{\Omega}|W(u^{\epsilon}(t))-W(u(t))|dx  \\
		&\leq |\Omega|\eta+\max_{y\in \R^{m}}|\nabla W(y)|\cdot \max_{t\in [0, T]}||u^{\epsilon}-u||_{L^{2}(\Omega)}|\Omega|^{\frac{1}{2}}\\
		\end{aligned}
		\end{equation*}
		where we have made use of $\eqref{uniformbound}$, Lipschitz continuity of $W$, and H\"older's inequality. Thus, from the fact that $u^\epsilon \to u$ in $C^0([0,T];L^2(\Omega))$, we can deduce that, up to a subsequence, $W_{\epsilon}(u^\epsilon) \to W(u)$ in $C^0([0,T];L^1(\Omega))$. The energy inequality for $u$ follows passing to the limit in $\eqref{Eestimate1}$.
		
		\item $u(0, x)=u_{0}$, $u_{t}(0,x)=v_{0}$
		
		Since $u^\epsilon \to u$ in $C^0([0,T];L^2(\Omega))$ and $u_{t}^{\epsilon} \to u_{t}$ in $C^{0}([0, T]; H^{-s}(\Omega))$, we have $u(0, x)=u_{0}$ and $u_{t}(0,x)=v_{0}$.
	\end{itemize}
\end{proof}
\begin{remark}
	Assume to have more regular initial data, i.e. $u_{0}\in \tilde{H}^{2s}(\Omega)$ and $v_{0}\in \tilde{H}^{s}(\Omega)$. For these data, the weak solution $u$ of $\eqref{eq:smlwaves}$ constructed in Theorem \ref{thm:main1} is energy preserving. Indeed, by \cite[Theorem 3(ii)]{MCMO}, the approximate solutions $u^\epsilon$ turn out to be more regular and energy preserving. Moreover, by using the uniform boundedness of $W_\epsilon$ and $\nabla W_\epsilon$, one can show that the velocity of the approximate solutions, namely $(u^{\epsilon}_{t})_{\epsilon}$, is uniformly bounded in $W^{1,\infty}(0, T; \tilde{H}^{s}(\Omega))$. This implies in the limit that $u_{t}\in W^{1,\infty}(0, T; \tilde{H}^{s}(\Omega))$, which in turn gives rise to the energy conservation of $u$ by using suitable test functions: by substituting the test function $\phi(t,x)=I_{[t_{1},\, t_{2}]}(t)\cdot u_{t}(t,x)$ in equality \eqref{eq:eqweak}, where $0\leq t_{1}<t_{2}\leq T$, and $I_{[t_{1},\, t_{2}]}$ is the indicator function on the time interval $[t_{1}, t_{2}]$, we obtain that
	\begin{equation*}
	\begin{aligned}
	&\int_{t_{1}}^{t_{2}}\langle u_{tt}(t),u_{t}(t)\rangle dt+\int_{t_{1}}^{t_{2}}[u(t),u_{t}(t)]_{s}dt+\int_{t_{1}}^{t_{2}}\int_{\Omega}\nabla_{u}W(u(t,x))u_{t}(t,x)dxdt=0\\
	&\Longleftrightarrow \int_{t_{1}}^{t_{2}} \frac{dE(u(t))}{dt}dt=0\\
	&\Longleftrightarrow E(u(t_{1}))=E(u(t_{2})),
	\end{aligned}
	\end{equation*}
	which proves that $E$ is constant inside the interval $[0, T]$.
\end{remark}

\section{The case of discontinuous $\nabla W$}\label{case2}
We consider in this section $W \in C(\R^{m})$ defined as
\begin{equation}\label{eq:discW}
W(y) =
\begin{system}
& |y|^2 &\quad& \text{ if } y\in \overline{\mathbf{B}(0, 1)} \\
& 1     &\quad& \text{ if } y\notin \mathbf{B}(0, 1) \\
\end{system}
\end{equation}
where $\mathbf{B}(0, 1)=\lbrace \, y\in \R^{m} \, | \, |y|< 1 \, \rbrace$, $\overline{\mathbf{B}(0, 1)}=\lbrace \, y\in \R^{m} \, | \, |y|\leq 1 \, \rbrace$. This potential designates $\partial \mathbf{B}(0, 1)$ as the set of critical states serving as boundary of the adhesive dynamics: as in \cite{CocliteFlorioLigaboMaddalena17}, this corresponds to model the adhesive contribution through a sharply discontinuous behaviour (adhesion inside $\mathbf{B}(0, 1)$, no adhesion outside).

Looking back at Definition \ref{def:weak}, we notice that the sharp discontinuity of $\nabla W$ on $\partial \mathbf{B}(0, 1)$ immediately jeopardizes the well-posedness of equality \eqref{eq:eqweak}: indeed, the term $\nabla_{u} W(u(t))$ is in principle not well-defined whenever $u(t) \in \partial \mathbf{B}(0, 1)$. One possible fix would be to arbitrarily choose a-priori a value for $\nabla W$ on the discontinuity set, but doing so invalidates any attempt to prove existence of weak solutions via an approximating approach. This is illustrated in the following example, where we consider for simplicity a Neumann problem in order to be able to write explicitly some approximate solutions and highlight why using \eqref{eq:eqweak} may not be adequate.
\begin{example}\label{exa}
	Consider the $1$-dimensional problem
	\begin{equation}\label{eq:smlwavesexa}
	\begin{system}
	& u_{tt}  - u_{xx} +W'(u)= 0       &\quad&\text{in } (0,T) \times (0, L)    \\
	& u_{x}(t,0) =u_{x}(t,L)= 0        &\quad&\text{in } [0,T]             		\\
	& u(0,x) = 1                       &\quad&\text{in } [0, L]                 \\
	& u_t(0,x) = 0                     &\quad&\text{in } [0, L]                 \\
	\end{system}
	\end{equation}
	for
	\begin{equation*}
	W(u)=
	\begin{system}
	& u^{2}	&\text{if }& |u|\leq 1 	\\
	& 1		&\text{if }& |u|> 1	\\
	\end{system}
	\quad \text{and} \quad
	W'(u)=
	\begin{system}
	& 2u	&\text{if }& |u|\leq 1 	\\
	& 0		&\text{if }& |u|> 1. 	\\
	\end{system}
	\end{equation*}
	Notice how we choose to set $W'(\pm 1) = \pm 2$. Consider now the sequence of approximate potentials $W_\epsilon$ with
	\begin{equation*}
	W_{\epsilon}'(u)=
	\begin{system}
	& (2-\epsilon)u								 &\text{if }& |u|\leq 1 				 \\
	& \frac{2-\epsilon}{\epsilon}(1+\epsilon-u)  &\text{if }& 1\leq u \leq 1+\epsilon 	 \\
	& \frac{\epsilon-2}{\epsilon}(1+\epsilon+u)  &\text{if }& -1-\epsilon \leq u \leq 1  \\
	& 0  										 &\text{if }& |u| \geq 1+\epsilon .		 \\
	\end{system}
	\end{equation*}
	One can easily show that $u^{\epsilon}(t,x) = 1+\epsilon$ solves the approximate problems
	\begin{equation*}
	\begin{system}
	& u^{\epsilon}_{tt}  - u^{\epsilon}_{xx} +W'_{\epsilon}(u^{\epsilon})= 0  &\quad&\text{in } (0,T) \times (0, L) 	\\
	& u^{\epsilon}_{x}(t,0) =u^{\epsilon}_{x}(t,L)= 0                         &\quad&\text{in } [0,T]                   \\
	& u^{\epsilon}(0,x) = 1+\epsilon                          				  &\quad&\text{in } [0, L]                  \\
	& u^{\epsilon}_t(0,x) = 0                         					      &\quad&\text{in } [0, L].                 \\
	\end{system}
	\end{equation*}
	These approximate solutions $(u^{\epsilon})_{\epsilon}$ converge to the constant function $1$ in $C([0,T] \times [0, L])$, satisfy \eqref{eq:eqweak}, but when attempting to pass to the limit we have
	\[
	\lim_{\epsilon \to 0} \int_\Omega W'_\epsilon(u^\epsilon(t)) \phi(t) \,dxdt = 0,
	\]
	while $W'(1) = 2$. Hence, the limit does not satisfy \eqref{eq:eqweak}, and generally we cannot pass to the limit in any definition of weak solution involving \eqref{eq:eqweak}. Indeed, the general lack of information on the distribution of the values of the approximate solutions $u_{\epsilon}$ around the critical states of $\nabla W$ prevents us from providing direct proofs by approximation. Hence, weaker notions of solutions are needed.
\end{example}

By Example \ref{exa}, the notion of weak solution provided by Definition \ref{def:weak} is not well-suited when dealing with potentials with discontinuous gradients. A restrictive result can however be provided, under the assumption of small initial data (i.e., when the troublesome region is completely avoided).
\begin{theorem}\label{thm:main2} Consider $2s>d$, $W$ as defined in \eqref{eq:discW} and assume that
	\begin{equation}\label{initaldata}
	||u_{0}||_{\Hst(\Omega)}\leq \epsilon_{1},\  ||v_{0}||_{L^{2}(\Omega)}\leq \epsilon_{2}
	\end{equation}
	for sufficiently small $\epsilon_{1}$, $\epsilon_{2}$. Then, there exists a weak solution of problem $\eqref{eq:smlwaves}$ in the sense of Definition $\ref{def:weak}$ with
	\begin{equation}\label{lessone}
	|u(x,t)|< 1 \quad \text{for all $(t, x)\in [0, T]\times \Omega$}
	\end{equation}
	and
	\begin{equation}\label{energyestimate}
	E(u(t))\leq E(u(0)) \quad \text{for any $t\in [0, T]$.}
	\end{equation}
\end{theorem}

\begin{proof}
	We repeat the approach used in the proof of Theorem \ref{thm:main1}: construct a family of non-negative potentials $(W_{\epsilon})_{\epsilon >0}$ in $C^{2}(\R)$ such that:
	\begin{itemize}
		\item[(i)] $W_{\epsilon}$ converges uniformly to $W$ in $\R^{m}$,
		\item[(ii)] $\nabla W_{\epsilon}$  converges pointwise to $\nabla W$ in $\R^{m} \setminus \partial \mathbf{B}(0, 1)$, $\nabla W_{\epsilon}$ converges uniformly to $\nabla W$ in $\mathbf{B}(0, 1)$, $\nabla W_{\epsilon}$ is uniformly bounded in $\R^{m}$,
		\item[(iii)] $\nabla W_{\epsilon}$ is Lipschitz for each $\epsilon$.
	\end{itemize} 
	For each $\epsilon >0$ there exists a weak solution $u_{\epsilon}$ in the sense of Definition $\ref{def:weak}$ corresponding to $W_{\epsilon}$ with initial data $u_{0}, v_{0}$ such that for any $t \in [0,T]$ one has 
	\begin{equation}\label{Estimate1}
	\frac12 ||u^\epsilon_{t}(t)||^{2}_{L^{2}(\Omega)}+\frac12 [u^\epsilon(t)]_{s}^2+||W_\epsilon(u^\epsilon(t))||_{L^{1}(\Omega)} \leq \frac12 ||v_0||^{2}_{L^{2}(\Omega)}+\frac12 [u_0]_{s}^2+||W_{\epsilon}(u_0)||_{L^{1}(\Omega)}.
	\end{equation}
	
	\medskip
	\noindent
	Since $W_{\epsilon}$ converges uniformly to $W$ in $\R^{m}$, for sufficiently small $\epsilon$ we have 
	\begin{equation}
	|W_{\epsilon}(y)-W(y)|\leq \epsilon_{3}
	\end{equation}
	for any $y\in \R^{m}$ and $\epsilon_{3} > 0$ fixed. This fact combined with $\eqref{initaldata}$  implies that
	\begin{equation}\label{Estimate2}
	||W_{\epsilon}(u_0)||_{L^{1}(\Omega)}\leq ||W(u_0)||_{L^{1}(\Omega)}+ \epsilon_{3}|\Omega|\leq |\Omega| \epsilon_{1}^{2}+ \epsilon_{3}|\Omega|.
	\end{equation}
	Thus, combining $\eqref{Estimate1}$ with estimates in $\eqref{initaldata}$ and $\eqref{Estimate2}$, we obtain that 
	\begin{align}\label{Estimate3}
	E(u^\epsilon(t)) &= \frac12 ||u^\epsilon_{t}(t)||^{2}_{L^{2}(\Omega)}+\frac12 [u^\epsilon(t)]_{s}^2+||W_\epsilon(u^\epsilon(t))||_{L^{1}(\Omega)} \leq C(\epsilon_{1}, \epsilon_{2}, \epsilon_{3}, \Omega)
	\end{align}
	for any $t \in [0,T]$.
	On the other hand, we have
	\begin{equation}
	\begin{aligned}
	||u^\epsilon(t)-u^\epsilon(0)||^{2}_{L^{2}(\Omega)}&=\int_{\Omega}\left|\int_{0}^{t}u^{\epsilon}_{t}(s,x)ds\right|^{2}dx \leq t\int_{\Omega}\int_{0}^{t}|u^{\epsilon}_{t}(s,x)|^{2}dsdx \\
	&\leq T\int_{0}^{t}\int_{\Omega}|u^{\epsilon}_{t}(s,x)|^{2}dxds \leq 2T^{2}C(\epsilon_{1}, \epsilon_{2}, \epsilon_{3}, \Omega)
	\end{aligned}
	\end{equation}
	where we have made use of Jensen's inequality and Fubini's theorem. Hence,
	\begin{equation}\label{Estimate4}
	||u^{\epsilon}(t)||_{L^{2}(\Omega)}\leq ||u^{\epsilon}(0)||_{L^{2}(\Omega)}+T\sqrt{2C(\epsilon_{1}, \epsilon_{2}, \epsilon_{3}, \Omega)} \leq \epsilon_{1}+T\sqrt{2C(\epsilon_{1}, \epsilon_{2}, \epsilon_{3}, \Omega)}
	\end{equation}
	for all $t\in [0, T]$.
	So, from the estimates $\eqref{Estimate3}$ and $\eqref{Estimate4}$ we obtain that
	\begin{equation}
	||u^{\epsilon}(t)||_{\Hst(\Omega)} \leq {C}(\epsilon_{1}, \epsilon_{2}, \epsilon_{3}, T, \Omega).
	\end{equation}
	Since $2s>d$, by means of the Sobolev embedding from $\Hst(\Omega)$ into the space $C^{0}(\R^{d};\R^{m})$ (see Lemma \ref{embedding}), we obtain
	\begin{equation}
	||u^{\epsilon}(t)||_{C^{0}(\Omega;\R^{m})}\leq C ||u^{\epsilon}(t)||_{\Hst(\Omega)} \leq {C}(\epsilon_{1}, \epsilon_{2}, \epsilon_{3}, T, \Omega),
	\end{equation}
	for all $t\in [0, T]$, where ${C}(\epsilon_{1}, \epsilon_{2}, \epsilon_{3}, T, \Omega)$ is decreasing as soon as $\epsilon_{1}, \epsilon_{2}, \epsilon_{3}$ are decreasing. Thus, for any small $\eta > 0$, by choosing $\epsilon_{1}, \epsilon_{2}, \epsilon_{3}$ small enough one has
	\begin{equation}
	|u^{\epsilon}(x,t)|\leq 1-\eta
	\end{equation}
	for any $(t,x)\in [0, T]\times \Omega$. Since approximate solutions never enter the discontinuity region of the gradient $\nabla W$, one can then repeat the same steps as in the proof of Theorem \ref{thm:main1} to pass to the limit along the sequence $(u^\epsilon)_\epsilon$ and obtain a weak solution satisfying \eqref{lessone} and \eqref{energyestimate}.
\end{proof}
\medskip
\noindent
As the above discussion made clear, in order to be able to handle problems with a discontinuity of the adhesive glue layer, i.e. discontinuities in $\nabla W$, a more robust notion of solution is needed. An immediate follow-up would be to consider for instance solutions in the sense of differential inclusions, e.g.
\[
u_{tt} + (-\Delta)^su \in -\partial W(u)
\]
(see for example $\cite{Chang}$, and references therein) or in the sense of Young measures, but we postpone such discussion to future works.
\section*{Acknowledgements}
The authors are partially supported by GNAMPA-INdAM. We thank the anonymous referees for their remarks that helped to improve the presentation of the note.

\end{document}